\documentclass{article}
\usepackage{mathrsfs}
\usepackage{amsmath,amscd,amsthm}
\usepackage{amsmath}

\usepackage{amsthm}

\input{amssym.tex} 

\newtheorem{Theorem}{Theorem}[section]

\newtheorem{Lemma}[Theorem]{Lemma}

\newtheorem{Remark}[Theorem]{Remark}

\newtheorem*{AssumptionT'}{Assumption T${}'$}
\newtheorem*{AssumptionT''}{Assumption T${}''$}

\def\eqref#1{(\ref{#1})}

\def\<<{\prec}

\def\a{\alpha}
\def\b{\beta}

\def\eps{\varepsilon}

\def\1{\mathbbold{1}}
\def\mathbbold{}
\def\mathbb{\Bbb}

\begin{document}
\bibliographystyle{plain}

\title {Polynomial loss of memory for maps of the interval with a neutral fixed point}

\author{Romain Aimino\thanks {Aix Marseille Universit\'e, CNRS, CPT, UMR 7332, 13288 Marseille, France and
Universit\'e de Toulon, CNRS, CPT, UMR 7332, 83957 La Garde, France.
e-mail:$<$aimino@cpt.univ-mrs.fr$>$.}
\and Huyi Hu \thanks{Mathematics Department, Michigan State University,
East Lansing, MI 48824, USA.
e-mail: $<$hu@math.msu.edu$>$.}
\and Matt Nicol \thanks{Department of Mathematics,
University of Houston,
Houston Texas,
USA. e-mail: $<$nicol@math.uh.edu$>$.}
\and Andrei T\"or\"ok \thanks{Department of Mathematics,
University of Houston,
Houston Texas,
USA. e-mail: $<$torok@math.uh.edu$>$ and Institute of Mathematics of
  the Romanian Academy, P.O. Box 1--764, RO-70700, Bucharest, Romania.}
\and Sandro Vaienti
\thanks{Aix Marseille Universit\'e, CNRS, CPT, UMR 7332, 13288 Marseille, France and
Universit\'e de Toulon, CNRS, CPT, UMR 7332, 83957 La Garde, France.
e-mail:$<$vaienti@cpt.univ-mrs.fr$>$.}}

\maketitle

\begin{abstract}
We give an example of a sequential dynamical system consisting of  intermittent-type maps which exhibits loss of memory with a polynomial rate of decay.
A uniform bound holds for the upper rate of memory loss. The maps may be chosen in any sequence, and the bound holds for all compositions.
\end{abstract}

\section{Introduction}
\setcounter{equation}{0}

The notion of loss of memory for non-equilibrium dynamical systems was introduced in the 2009  paper by Ott, Stenlund and Young \cite{OSY}; they wrote:

 \small  Let $\rho_0$ denote an initial probability density w.r.t. a reference measure $m$, and suppose its time evolution is given by $\rho_t$.
 One may ask if these probability distributions retain memories of their past.
 We will say a system {\bf loses its memory in the statistical sense} if for two initial distributions $\rho_0$ and $\hat{\rho_0}$, $\int|\rho_t- \hat{\rho_t}|dm\rightarrow 0$.
 \normalsize

In~\cite{OSY} the rate of  convergence of the two densities was proved to be exponential for certain sequential dynamical systems composed
of one-dimensional piecewise expanding maps. Coupling was the technique used for the proof. The same technique was successively applied to time-dependent Sinai billiards with moving scatterers by Stenlund, Young, and Zhang \cite{SYZ} and it gave again an exponential rate.  A different approach, using the Hilbert projective metric, allowed Gupta, Ott and T\"or\"ok \cite{GOT} to obtain exponential loss of memory for time-dependent multidimensional piecewise expanding maps. 

All the previous papers prove an  exponential loss of memory in the strong sense, namely
$$
\int |\rho_t-\hat{\rho_t}|dm \le C e^{-\alpha t}.
$$
In the invertible setting, Stenlund \cite{S} proves  loss of memory in the weak-sense for random composition of Anosov diffeomorphisms, namely
$$
\left|\int f\circ{\mathcal T}_nd\mu_1-\int f\circ{\mathcal T}_n d\mu_2 \right| \le C
e^{-\alpha t}
$$
where $f$ is a H\"{o}lder observable, ${\mathcal T}_n$ denotes the composition of $n$ maps and  $\mu_1$ and $\mu_2$ are two probability measures absolutely continuous with respect to the Riemannian volume whose densities are H\"older. It is easy to see that loss of memory in the strong  sense implies loss of memory in the weak sense, for densities in the corresponding function spaces and $f\in L^\infty_m$.


 A natural question is: are there examples of time-dependent systems
 exhibiting loss of memory with a slower rate of decay, say polynomial,
 especially in the strong sense? We will construct such an example in this
 paper as a (modified) Pomeau-Manneville map:
 \begin{equation}\label{PM}
   T_{\a}(x)= \begin{cases}
     x + \frac{3^{\a}}{2^{1+\a}}x^{1+\a}, \ 0\le x \le 2/3\\
     3x-2, \ 2/3 \le x \le 1
   \end{cases} \qquad 0 < \alpha <1.
 \end{equation}
    We use this version of the Pomeau-Manneville intermittent map because the derivative is increasing on [0,1), where it is defined, and this allows us to simplify the exposition. We believe the result remains true for time-dependent systems comprised of  the {\em usual} Pomeau-Manneville maps, for instance the version studied in \cite{LSV}.\footnote{See the ``Note added in proof'' at the end of the paper.} We will refer quite often to \cite{LSV} in this note. As in \cite{LSV}, we will identify the unit interval $[0,1]$ with the circle $S^1$, in such a way the map becomes continuous. 

We will see in a moment how an initial density evolves  under composition with maps which are slight perturbations of (\ref{PM}). To this purpose we will  define the perturbations of the usual
    Pomeau-Manneville map that we will consider. 

     The perturbation will be defined by considering maps $T_{\b}(x)$ as
     above with $0< \b \le \a$. Note that $T_{\b}=T_{\a}$ on $2/3 \le x \le 1$. 
  The reference measure will be Lebesgue ($m$).
  If $0< \b_k \le \a$ is chosen, we denote by $P_{\b_k}$ the {\em Perron-Frobenius} (PF) transfer operator associated to the map $T_{\b_k}.$

Let us suppose $\phi, \ \psi$ are two observables in an appropriate (soon to be defined) functional space; then the  basic quantity that we have to control is
\begin{equation}\label{DC}
\int |P_{\b_n}\circ \cdots \circ P_{\b_1}(\phi)-P_{\b_n}\circ \cdots \circ P_{\b_1}(\psi)|dm.
\end{equation}
 Our goal is to show that  it decays polynomially fast and independently of
 the sequence $P_{\b_n}\circ \cdots \circ P_{\b_1}$: we stress that there
 is no probability vector to weight the $\b_k$.
 Note that, by the results of \cite{Sarig}, one cannot have in general a
 faster than polynomial decay, because that is the (sharp) rate when iterating a  single map $T_\beta$, $0< \beta <1$.

 In order to prove our
 result, Theorem~\ref{thm.main}, we will follow the strategy used in \cite{LSV} to get a polynomial
 upper bound (up to a logarithmic correction) for the correlation decay. We introduced there a perturbation
 of the transfer operator which was, above all,  a  technical tool to
 recover the loss of dilatation around the neutral fixed point by replacing
 the observable with its conditional expectation to a small ball around
 each point. It turns out that the same technique allows us to control the evolution of two densities under concatenation of maps if we can control the distortion of this sequence of maps. The control of distortion  will be, by the way, the major difficulty of this paper.

 Note that the convergence of the quantity (\ref{DC}) implies the decay of the non-stationary correlations, with respect to $m$:
$$\begin{aligned}
\left|\int \psi(x) \phi\circ T_{\b_n}\circ \cdots \circ T_{\b_1}(x) dm-\int \psi(x) dm \ \int \phi\circ T_{\b_n}\circ \cdots \circ T_{\b_1}(x) dm \right| \\ \le
\|\phi\|_{\infty} \left\|P_{\b_n}\circ \cdots \circ P_{\b_1}(\psi)-P_{\b_n}\circ \cdots \circ P_{\b_1} \left({\bf 1} \left(\int \psi dm\right)\right) \right\|_1 \end{aligned}
$$
 provided $\phi$ is essentially bounded and ${\bf 1}(\int \psi dm)$ remains in the functional space where the convergence of (\ref{DC}) takes place. In particular, this holds for $C^1$ observables, see Theorem~\ref{thm.main}. 

 Conze and Raugy \cite{CR} call the decorrelation described above  {\em decorrelation} for the {\em sequential dynamical system} $T_{\b_n}\circ \cdots \circ T_{\b_1}.$
Estimates on the rate of decorrelation (and the function space in which decay occurs)  are a key ingredient in the Conze-Raugy theory to establish central limit theorems for the sums $\sum_{k=0}^{n-1}\phi(T_{\b_k}\circ \cdots \circ T_{\b_1}x)$, after centering and normalisation. The question could be formulated in this way:
does the ratio
$$
\frac{\sum_{k=0}^{n-1}[\phi \circ T_{\b_k}\circ \cdots \circ T_{\b_1}(x)-\int \phi \circ T_{\b_k}\circ \cdots \circ T_{\b_1} dm] }{\|\sum_{k=0}^{n-1}\phi\circ T_{\b_k}\circ \cdots \circ T_{\b_1}\|_2}
$$
converge in distribution  to the normal law $\mathcal{N}(0,1)$?

It would be interesting to establish such a limit theorem for the sequential dynamical system constructed with our intermittent map (\ref{PM}). 
Besides the central limit theorem, other interesting questions could be considered for our sequential dynamical systems, for instance the existence of {\em concentration inequalities} (see the recent work \cite{AR} in the framework of the Conze-Raugy theory), and the existence of {\em stable laws}, especially  for perturbations of maps $T_{\alpha}$ with $\alpha>1/2$, which is the range for which the unperturbed map exhibits stable laws~\cite{GO}.

We said above that we did not choose the  sequence of maps $T_{\b}$ according to some probability distribution. A random dynamical system has been considered in the recent paper~\cite{BH} for similar perturbations of the usual Pomeau-Manneville map. To establish a correspondence with our work, let us say that those authors perturbed the map $T_{\alpha}$ by modifying again the slope, but taking this time finitely many values  $0 < \alpha_1 < \alpha_2 < \dots < \alpha_r \le 1$, with a finite discrete law. This random transformation has a unique stationary measure, and the authors  consider annealed correlations on the space of H\"older functions. They prove in~\cite{BH} that such annealed correlations decay polynomially at a rate bounded above by $n^{1-\frac{1}{\alpha_1}}$.

As a final remark, we would like to address the question of proving the loss of memory for intermittent-like maps, but with the sequence given by adding a varying constant to the original map, considered to act on the  unit circle (additive noise). This problem seems much harder and a possible strategy would be to consider induction schemes, as it was done recently in \cite{SS} to prove stochastic stability in the strong sense.

 \textbf{NOTATIONS.}  We will index the perturbed maps and transfer operators respectively as $T_{\b_k}$ and $P_{\b_k}$ with $0< \b_k\le \a$.  Since we will be interested  in concatenations like $P_{\b_n}\circ P_{\b_{n-1}}\circ\cdots \circ P_{\b_m}$  we will use equivalently  the following notations
 $$
 P_{\b_{n}}\circ P_{\b_{n-1}}\circ \cdots \circ P_{\b_m}= P_{n}\circ P_{n-1}\circ\cdots \circ P_m.
 $$
 We will see that very often the choice of $\b_k$ will be not important in the construction of the concatenation; in this case we will adopt the useful notations, where the exponent of the $P$'s is the number of transfer operators in the concatenation:
 $$
 P_{\b_{n}}\circ P_{\b_{n-1}}\circ\cdots \circ P_{\b_m}:=P_m^{n-m+1}$$$$ P_k^n=P_{k+n-1}\circ P_{k+n-2}\circ\cdots \circ P_{k}
 $$
In the same way, when we concatenate  maps we will use the notation $$T_m^{n-m+1} := T_n\circ T_{n-1}\circ\cdots \circ T_m$$ instead of $T_{\b_n}\circ T_{\b_{n-1}}\circ\cdots \circ T_{\b_m}.$

Finally, for any sequences of numbers $\{a_n\}$ and $\{b_n\}$, we will write
$a_n\approx b_n$
if $c_1b_n\le a_n\le c_2b_n$ for some constants $c_2\ge c_1>0$. The first derivative will be denoted as either $T'$ or $DT$ and the value of $T$ on the point $x$ as either $Tx$ or $T(x).$

 \section{The cone, the kernel, the decay}

 Thanks to a general theory by  Hu \cite{HH}, we know that the density $f$  of the absolutely continuous invariant measure of  $T_{\a}$ in the neighborhood of $0$  satisfies  $f(x)\le \mbox{constant} \ x^{-\a}$, where the value of the constant has an expression in terms  of the value of $f$ in the pre-image of $0$ different from $0$. We will construct  a cone which is preserved by the transfer operator of each $T_{\b}, 0<\b\le \a$, and the density of each $T_{\b}$  will be the only fixed point of a suitable subset of that cone.

  We   define the cone of  functions
    $$
    {\mathcal C_1}:= \{f\in C^0(]0,1]); \ f\ge 0; \ f \ \mbox{decreasing};\  X^{\a+1}f \ \mbox{increasing}\} 
    $$
    where $X(x)=x$ is the identity function.
    \begin{Lemma}
    The cone ${\mathcal C}_1$ is {\em invariant} with respect to the operators $P_{\b}, \ 0< \b \le \a < 1.$
    \end{Lemma}
    \begin{proof}
    Put $T_{\b}^{-1}(x)=\{y_1, y_2\}, y_1<y_2;$ put also $\chi_{\b}=\frac{3^{\b}y_1^{\b}}{2^{1+\b}}.$ Then a direct computation shows that
  $$
  X^{\a+1}P_{\b}f(x)= \frac{f(y_1)y_1^{\a+1}(1+\chi_{\b})^{\a+1}}{1+(1+\b)\chi_\b}+f(y_2)\left(\frac{3y_2-2}{y_2}\right)^
  {\a+1}\frac{y_2^{\a+1}}{3}.
  $$
  The result now follows since the maps  $x\rightarrow x^{\a+1}f(x)$, $x\rightarrow \chi_{\b}$, $x\rightarrow y_1$, $x\rightarrow y_2$ are increasing. The fact that $\a\ge \b$ implies the  monotonicity of $\chi\rightarrow  \frac{(1+\chi)^{\a+1}}{1+(1+\b)\chi}.$
    \end{proof}
    We now denote $m(f)=\int_0^1 f(x) dx$ and recall that for any $0<\b<1$ we have $m(P_{\b}f)=m(f).$


    \begin{Lemma}
      Given $0 < \a < 1$, the cone
      $$
      {\mathcal C}_2:=\{f\in {\mathcal C}_1\cap L^1_m; \ f(x)\le a x^{-\a}\ m(f)\}
      $$
      is preserved by all the operators $P_{\b}, \ 0< \b  \le \a$, provided
      $a$ is large enough.
   \end{Lemma}
   \begin{proof}
   Let us suppose that $\int_0^1 fdx=1;$ then we look for a constant $a$ for which $P_{\b}f(x)\le a x^{-\a}.$ Using the notations in the proof of the previous Lemma and remembering that $x^{\a+1}f(x)\le f(1)\le \int_0^1 fdx=1$,  we get $$ \begin{aligned}
   P_{\b}f(x)=\frac{f(y_1)}{T'_{\b}(y_1)}+\frac{f(y_2)}{T'_{\b}(y_2)} & \le
   \frac{ay_1^{-\a} }{T'_{\b}(y_1)}+\frac{y_2^{-\a-1}}{T'_{\b}(y_2)}  \\ &= \left\{\left(\frac{x}{y_1}\right)^{\a}\frac{1}{T'_{\b}(y_1)}+\frac1a \frac{x^{\a}}{y_2^{\a+1}T'_{\b}(y_2)}\right\}a x^{-\a}, \end{aligned}
   $$
   but
   $$ \begin{aligned}
   \left(\frac{x}{y_1}\right)^{\a}\frac{1}{T'_{\b}(y_1)}+\frac1a \frac{x^{\a}}{y_2^{\a+1}T'_{\b}(y_2)} & \le \frac{(1+\chi_{\b})^{\a}}{1+(1+\b)\chi_{\b}}+\frac1a (\frac32y_1)^{\a-\b}\chi_\b (1+\chi_\b)^{\a} \\ & \le
   \frac{(1+\chi_{\b})^{\a}}{1+(1+\b)\chi_{\b}}+\frac1a (\frac32)^{\a}\chi_{\b} , \ (*) \end{aligned}
   $$
   where the last step is justified by the fact that $\b\le \a$ and $0\le \chi_\b\le 1/2$. By taking the common denominator one gets
   $$
   (*)\le \frac{1+\{\b+[(\a-\b)+2^{\a}a^{-1}(\b+2)]\}\chi_{\b}}{1+(1+\b)\chi_{\b}}.
   $$
   We get the desired result if $(\a-\b)+2^{\a}a^{-1}(\b+2)\le 1$, which is satisfied whenever
   $$
   a\ge \frac{2^{\a}(2+ \a)}{1-\a}.
   $$\end{proof}

   \begin{Remark}\label{R1}

The preceding two lemmas imply the following properties which will be used later on.
\begin{enumerate}
\item $\forall f\in {\mathcal C}_2, \ \inf_{x\in [0,1]}f(x)=f(1)\ge \min\{a; \ [\frac{\a(1+\a)}{a^{\a}}]^{\frac{1}{1-\a}}\}m(f)$.
    \item \label{R1_item2} For any concatenation $P_{1}^m=P_m\circ \cdots \circ P_1$ we have $$P_{1}^m{\bf 1}(x)\ge \min\{a; \ [\frac{\a(1+\a)}{a^{\a}}]^{\frac{1}{1-\a}}\}.$$
 \end{enumerate}
 \end{Remark}

 See the proof of Lemma 2.4 in \cite{LSV} for the proof of  the first item, the second  follows immediately from the first.

 \begin{Remark}
 Using the previous Lemmas it is also possible to prove the existence of the  density in  ${\mathcal C}_2$ for the unique a.c.i.m. by using the same argument as in Lemma 2.3 in \cite{LSV}.
 \end{Remark}

We now take $f\in {\mathcal C}_2$ and define the {\em averaging operator} for $\eps>0$:
$$
      \mathbb{A}_{\eps}f(x):=\frac{1}{2\eps}\int_{B_{\eps}(x)} f dm
      $$
 where $B_r(x)$ denotes the ball of radius $r$ centered at the point $x\in S^1$, and  define a new {\em perturbed transfer operator} by
      $$
      \mathbb{P}_{\eps,m}:= P_{m}^{n_{\eps}}\mathbb{A}_{\eps}= P_{\b_{m+n_{\eps}-1}}\circ \cdots \circ P_{\b_m}\mathbb{A}_{\eps}
      $$
      where $n_{\eps}$ will be defined later on.
       It is very easy to see that
    \begin{Lemma}\label{DO}
      For $f\in {\mathcal C_2}$
      $$
      \|\mathbb{P}_{\eps,m}f-P_{m}^{n_{\eps}}f\|_1\le c \|f\|_1 \eps^{1-\a}
      $$
      where $c$ is  independent of $\b$.
      \end{Lemma}
\begin{proof}
By linearity and contraction of the operators $P_{\b}$ we bound the left hand side  of the quantity in the statement of the lemma by $\int |\mathbb{A}_{\eps}f-f|dx$ and this quantity gives the prescribed bound as in Lemma 3.1 in \cite{LSV}.
\end{proof}

It is straightforward to get the following representation for the operator $ \mathbb{P}_{\eps, m}:$
$$
 \mathbb{P}_{\eps,m}f(x)=\int_0^1 K_{\eps, m}(x,z)f(z)dz
$$
 where
$$
 K_{\eps,m}(x,z):=\frac{1}{2\eps}P_{m}^{n_{\eps}}{\bf 1}_{B_{\eps}(z)}(x).
$$

We now observe that standard computations (see for instance Lemma 3.2 in \cite{LSV}), allows us to show that the preimages $a_n^{\a}:=T^{-n}_{\a,1}1$ verify $a_n^{\a} \approx \frac{1}{n^{\frac{1}{\a}}};$ here $T^{-1}_{\a,1}$ denotes the left pre-image of $T^{-1}_{\a}$, a notation which we will also use later on. Those points are the boundaries of a countable Markov partition and they will play a central role in the following computations; notice that the factors $c_1, c_2$ in the bounds $c_1\frac{1}{n^{\frac{1}{\a}}}\le a_n^{\a} \le c_2\frac{1}{n^{\frac{1}{\a}}}$ depend on $\a$ (and therefore on $\b$), but we will only use the $a_n$ associated to the exponent $\a$; in particular we will denote by  $c_{\a}$ the constant $c_2$ associated to $T_{\a}; $ the dependence on $\a$, although implicit, will not play any role in the following.

We will prove in the next section the following  important fact.

\begin{itemize}\item {\em Property ({\bf P})}. \  There exists $\gamma>0$ and $n_\epsilon = \mathcal{O}(\epsilon^{- \alpha})$ such that for all $\eps>0$, $x,z\in [0,1]$ and for any sequence $\b_m,\cdots,\b_{m+n_{\eps}-1}$, one has
$$
 K_{\eps, m}(x,z)\ge \gamma.
$$
\end{itemize}

 We now show how the positivity of the kernel implies the main result of this paper.

 \begin{Theorem}\label{thm.main}Suppose  $\psi, \phi$ are in ${\mathcal C}_2$ for some $a$ with equal expectation $\int \phi dm= \int \psi dm$. Then for any  $0<\a<1$ and for any  sequence $T_{\b_1},\cdots, T_{\b_n}$, $n>1$, of maps of Pomeau-Manneville type (\ref{PM}) with $0 < \b_k\le \a, \ k\in [1,n]$, we have
 $$
 \int |P_{\b_n}\circ\cdots\circ P_{\b_1}(\phi)-P_{\b_n}\circ\cdots\circ P_{\b_1}(\psi)|dm \le C_{\a} (\|\phi\|_1+\|\psi\|_1)n^{-\frac{1}{\a}+1}(\log n)^{\frac{1}{\a}},
 $$
 where the constant $C_{\a}$ depends only on the map $T_{\a}$, and $\|\cdot\|_1$ denotes the $L^1_m$ norm.

 A similar rate of decay holds  for  $C^1$ observables $\phi$ and $\psi$ on $S^1$; in this case the rate of decay has an upper bound given by
 $$
  C_{\a} \ \mathcal{F}(\|\phi\|_{C^1}+\|\psi\|_{C^1})n^{-\frac{1}{\a}+1}(\log n)^{\frac{1}{\a}}
 $$
 where the function $\mathcal{F}: \mathbb{R}\rightarrow \mathbb{R}$ is affine.
 \end{Theorem}

\begin{Remark}One can ask what happens if we relax the assumption that all $\beta_n$ must lie in an interval $[0, \alpha]$ with $0 < \alpha < 1$. For instance, if the sequence $\beta_n$ satisfies $\beta_n < 1$ and $\beta_n \to 1$, does the quantity $\|P_1^n \phi - P_1^n \psi \|_1$ go to $0$ for all $\phi, \psi$ in $C^1$ with $\int \phi = \int \psi$? Similarly, what can we say when $\beta_n \to 0$? It follows from our main result that the decay rate of $\| P_1^n \phi - P_1^n \psi \|_1$ is superpolynomial, but can we get more precise estimates for particular sequences $\beta_n$, like $\beta_n = n^{- \theta}$ or $\beta_n = e^{-c n^\theta}$, $\theta > 0$?  We can also ask whether there is, in the case $\beta_n  \in [0, \alpha]$ covered by our result, an elementary proof for the decay to zero (without rate) of $\|P_1^n \phi - P_1^n \psi \|_1$.
\end{Remark}
\begin{proof}[Proof of Theorem~\ref{thm.main}]
We begin to  prove the first part of the theorem for $\mathcal{C}_2$ observables. We write $n=kn_{\eps}+m$ with $m < n_\epsilon$. We add and subtract  to the difference in the integral a term composed by the product of the first $m$ usual PF operators and the product of $k$ averaged operator $\mathbb{P}_{\eps}$, each composed by $n_{\eps}$ random PF operators;  precisely we use the notation introduced above to get:
 $$ \begin{aligned}
 (LM) & := \int |P_{\b_n}\circ\cdots\circ P_{\b_1}(\phi)-P_{\b_n}\circ\cdots\circ P_{\b_1}(\psi)|dm \\ & =
  \int |P_{1}^n(\phi)-\mathbb{P}_{\eps, m+1+(k-1)n_{\eps}} \circ \cdots \circ \mathbb{P}_{\eps, m+1}P_{1}^m(\phi) \\ &
 + \mathbb{P}_{\eps, m+1+(k-1)n_{\eps}} \circ \cdots \circ \mathbb{P}_{\eps, m+1}P_{1}^m(\phi)
  \\ & -\mathbb{P}_{\eps, m+1+(k-1)n_{\eps}} \circ \cdots \circ \mathbb{P}_{\eps, m+1}P_{1}^m(\psi)
 \\ &  +\mathbb{P}_{\eps, m+1+(k-1)n_{\eps}} \circ \cdots \circ \mathbb{P}_{\eps, m+1}P_{1}^m(\psi)-P_{1}^n(\psi)|dm. \end{aligned}
 $$
Thus
$$ \begin{aligned}
(LM) & \le \|P_{1}^n(\phi)-\mathbb{P}_{\eps, m+1+(k-1)n_{\eps}} \circ \cdots \circ \mathbb{P}_{\eps, m+1}P_{1}^m(\phi)\|_1 \\ & +
 \|P_{1}^n(\psi)-\mathbb{P}_{\eps, m+1+(k-1)n_{\eps}} \circ \cdots \circ \mathbb{P}_{\eps, m+1}P_{1}^m(\psi)\|_1 \\ &+
\|\mathbb{P}_{\eps, m+1+(k-1)n_{\eps}} \circ \cdots \circ \mathbb{P}_{\eps, m+1}P_{1}^m(\phi-\psi)\|_1. \end{aligned}
$$
We now treat the first term $I$ in $\phi$ on the right hand side ( the terms in $\psi$ being equivalent), and we consider the last term $III$ after that. We thus have:
$$
I= \|P_{m+1+(k-1)n_{\eps}}^{n_{\eps}}\cdots P_{ m+1}^{n_{\eps}}P_{1}^m(\phi)-
\mathbb{P}_{\eps,  m+1+(k-1)n_{\eps}} \circ \cdots \circ \mathbb{P}_{\eps, m+1}P_{1}^m(\phi)\|_1.
$$
To simplify the notations we put
  $$
    \begin{cases}
     \mathbb{R}_1:=\mathbb{P}_{\eps,m+1},\\
\vdots\\
  \mathbb{R}_k:=\mathbb{P}_{\eps, m+1+(k-1)n_{\eps}},
    \end{cases}
    $$
and
$$
    \begin{cases}
    \mathbb{Q}_1:=P_{m+1}^{n_{\eps}},\\
\vdots\\
  \mathbb{Q}_k:=P_{m+1+(k-1)n_{\eps}}^{n_{\eps}},
    \end{cases}
    $$
which reduce the above inequality to
$$
I= \| (\mathbb{Q}_k\cdots \mathbb{Q}_1-\mathbb{R}_k\cdots \mathbb{R}_1)P_{1}^m(\phi)\|_1.
$$
By induction we can easily see that
$$\mathbb{R}_k\cdots \mathbb{R}_1 -
\mathbb{Q}_k\cdots \mathbb{Q}_1=\sum_{j=1}^k\prod_{l=0}^{k-j-1}\mathbb{R}_{k-l}(\mathbb{R}_j-\mathbb{Q}_j)\prod_{l=0}^{j-1}\mathbb{Q}_{j-l-1}
$$
with $\mathbb{R}_{-1}={\bf 1}$ and $\mathbb{Q}_0={\bf 1}$; by setting $\phi_m:=P_{1}^m(\phi)$ and $\tilde{\phi}_m=P_1^m(\phi-\psi)$, we have therefore to bound by the quantity
$$
\sum_{j=1}^k\|\prod_{l=0}^{k-j-1}\mathbb{R}_{k-l}(\mathbb{R}_j-\mathbb{Q}_j)\prod_{l=0}^{j-1}\mathbb{Q}_{j-l-1}\phi_m\|_1.
$$
We now observe that $\mathbb{Q}_{j-l-1}\phi_m\in {\mathcal C_2}$; moreover $\|\mathbb{R}_{m}g\|_1\le \|g\|_1$\ $\forall g\in {\mathcal C_2}, \ 1\le m\le k$, since $\mathbb{R}_m$ is a concatenation of transfer operators and the averaging map $\mathbb{A}_{\eps}$ which are all contractions on $L^1$.
Then we finally get, by invoking also Lemma \ref{DO},
$$ \begin{aligned}
I & \le \| \mathbb{Q}_k\cdots \mathbb{Q}_1\phi_m-\mathbb{R}_k\cdots \mathbb{R}_1\phi_m\|_1 \\ & \le \sum_{j=1}^k c \|\phi_m\|_1\eps^{1-\a}\le c k \|\phi\|_1\eps^{1-\a}. \end{aligned}
$$
We now look at the third term $III$ which could be written as, by using the simplified notations introduced above: $III=\|\mathbb{R}_k\cdots \mathbb{R}_1\tilde{\phi}_m\|_1.$ By using Property ({\bf P}) and by applying the same arguments as in the footnote $6$ in \cite{LSV}, one gets
$$
\|\mathbb{R}_k\cdots \mathbb{R}_1\tilde{\phi}_m\|_1  \le e^{-\gamma k}\|\phi-\psi\|_1.
$$
In conclusion we get
$$\begin{aligned}
(LM)& \le c k \eps^{1-\a}(\|\phi\|_1+\|\psi\|_1)+e^{-\gamma k}(\|\phi\|_1+\|\psi\|_1) \\ & \le
\left(c \frac{n}{n_{\eps}}\eps^{1-\a}+e^{\gamma}\ e^{-\gamma \frac{n}{n_{\eps}}} \right) (\|\phi\|_1+\|\psi\|_1)\le C_{\a} \ (\|\phi\|_1+\|\psi\|_1) n^{1-\frac{1}{\a}} (\log n)^{\frac{1}{\a}} \end{aligned}
$$
having chosen  $\eps= n^{-\frac{1}{\a}}\left(\log n^{(\frac{1}{\a}-1) \kappa}\right)^{\frac{1}{\a}}$, for a conveniently chosen $\kappa$.

In order to prove the second part of the theorem for $C^1$ observables, we invoke the same argument as at the end of the proof of Theorem 4.1 in \cite{LSV}.
We notice in fact that if $\psi \in C^1$ then we can choose $\lambda, \nu\in \mathbb{R}$ such that $\psi_{\lambda, \nu}(x)=\psi + \lambda x+ \nu\in \mathcal{C}_2,$ the dependence of the parameters with respect to the $C^1$ norm being affine.

For instance $\lambda$ and $\nu$ could be chosen in such a way to verify the following constraints:  $\lambda<-\|\psi'\|_{\infty}$; $\nu>\max\{ \frac{(1+\a)\|\psi\|_{\infty}+\|\psi'\|_{\infty}-\lambda(2+\a)}{1+\a}, \frac{1+a}{a-1}\|\psi\|_{\infty}-\frac{a\lambda}{2(a-1)}\}.$
\end{proof}

\section{Distortion: Proof of Property ({\bf P})}

The main technical problem is now to check the positivity of the kernel; we will follow closely the strategy of the proof of Proposition 3.3 in \cite{LSV}. We recall that
$$
2\eps \ K_{\eps,  m}(x, z )= P_{m}^{n_{\eps}}{\bf 1}_{J}(x)
$$
where $J=B_{\eps}(z)$ is an interval which we will take later on as a ball of radius $\eps$ around $z$.


 By iterating  we get (we denote with $T^{-1}_{l,k}, \ k=1,2,$ the two inverse branches of $T_l$):
$$ \begin{aligned}
2\eps \ K_{\eps,  m} & = \sum_{l_{n_{\eps}}}\cdots \sum_{l_{1}}\frac{{\bf 1}_J(T^{-1}_{1,l_1}\cdots T^{-1}_{n_{\eps},l_{n_{\eps}} }x)}{|T_1'(T^{-1}_{1,l_1}\cdots T^{-1}_{n_{\eps},l_{n_{\eps}} }x) T_2'(T^{-1}_{2,l_2}\cdots T^{-1}_{n_{\eps},l_{n_{\eps}} }x)\cdots T'_{n_{\eps}}(T^{-1}_{n_{\eps},l_{n_{\eps}}}x)|} \\ &= \sum_{l_{n_{\eps}}}\cdots \sum_{l_{1}}\frac{{\bf 1}_J(x_{n_{\eps}})}{|T'_1(x_{n_{\eps}}) T'_2(T_1x_{n_{\eps}})\cdots T'_{n_{\eps}}(T_{n_{\eps}-1}\cdots T_1 x_{n_{\eps}})|} \end{aligned}
$$
where $x_{n_{\eps}}=T^{-1}_{1,l_1}\cdots T^{-1}_{n_{\eps},l_{n_{\eps}} }x$ ranges over all points in the preimage of  $x\in T_{n_{\eps}}\circ \cdots \circ T_1 J.$
The quantity on the right hand side is bounded from below by
$$
2\eps \ K_{\eps,  m}\ge {\bf 1}_{T_{n_{\eps}}\circ \cdots \circ T_1(J)}(x)\ \inf_{z\in J}\frac{1}{|T'_1(z) T'_2(T_1z)\cdots T'_{n_{\eps}}(T_{n_{\eps}-1}\cdots T_1 z)|}.
$$

We have therefore to control the ratio

$$\inf_{z\in J}\frac{1}{|T'_1(z) T'_2(T_1z)\cdots T'_{m}(T_{m-1}\cdots T_1 z)|}$$
where $m$ is the time needed for an interval $J$ of length greater than $2 \epsilon$ to cover all the circle. We proceed as in the proof of Proposition 3.3 in \cite{LSV}.

We need to introduce first some notations. Recall that $a_n^{\alpha}$ is the sequence of the preimages of $1$ by the left branch of $T_\alpha$. We use similarly $a_n^{\beta}$ for $T_\beta$ and define $a_n^0$ as the infimum over all $\beta >0$ of $a_n^{\beta}$. Remark that $a_n^0$ is the sequence of the preimages of $1$ by the left branch of the map $T_0$ defined by
\begin{equation}
   T_{0}(x)= \begin{cases}
     \frac{3x}{2}, \ 0\le x \le 2/3\\
     3x-2, \ 2/3 \le x \le 1.
   \end{cases}
 \end{equation}

For $k \ge 1$, we define the sequence $a_n^k$ so that $a_0^k = 1$ and $a_n^k$ is the preimage of $1$ by $T_{k+1}^n$ the most at the left. In particular, $a_n^k$ is the preimage of $a_{n-1}^{k+1}$ by the left branch of $T_{k+1}$. Remark that $a_n^k$ is a decreasing sequence in $n$ and that $a_n^{0} \le a_n^k \le a_n^{\alpha}$.

We define the intervals $I_n^k = [a_{n+1}^k, a_n^k]$, which satisfy $T_{k+1}^n I_n^k = [\frac{2}{3}, 1]$. We also define $I_{n,+}^k = I_{n+1}^k \cup I_n^k = [a_{n+2}^k, a_n^k]$.

We define the intermittent region $I = [0, a_2^{0}]$ and the hyperbolic region $H = [a_2^{0}, 1]$.

Let $J$ be an interval of length $2 \epsilon$. We will iterate $J$ under the non-stationary dynamics until it covers the whole space, and will control the distortion in the meantime.

At time $k$, the iterate $K = T_1^k J$ verifies one of the following condition

\begin{enumerate}

\item \label{case:hyperb} $K \cap I = \emptyset$;

\item \label{case:inter_1} $K \cap I \neq \emptyset$, and $K$ contains at most one $a_{\ell}^k$, $\ell > 2$;

\item \label{case:inter_2} $K \cap I \neq \emptyset$, and $K$ contains more than one $a_{\ell}^k$, $\ell > 2$.

\end{enumerate}
~ \\

{\bf Case \ref{case:hyperb}.} Suppose we are in the situation \ref{case:hyperb}. Either one of the iterates of $K$ will cross the point $\frac 2 3$ where the maps are not differentiable, or it will fall in the situation \ref{case:inter_1} or \ref{case:inter_2}. Let $n \ge 1$ be the time spent before one of these situations occurs.

Since all maps are uniformly expanding on the hyperbolic region with uniformly bounded second derivatives, by standard computations, we have for all $a,b \in K$ : $$\frac{(T_{k+1}^n)'(a)}{(T_{k+1}^n)'(b)} \le \exp \left\{ \sum_{\ell=0}^{n - 1} \frac{ \sup_{\xi} |T_{k+n - \ell}'' \xi |}{\inf_\xi |T_{k+n - \ell}' \xi|} |T_{k+1}^{n+k -\ell - 1} a - T_{k+1}^{n+k -\ell - 1} b| \right\}.$$

Since $0 < \beta \le \alpha < 1$, the ratio $\frac{|T_\beta^{''} x|}{|T_\beta^{'} x|}$ and the quantity $|T_\beta^{'} x|$ are bounded from above uniformly in $\beta$ and $x \in H$ respectively by $D > 0$ and $0 < r < 1$. We then have

$$\frac{(T_{k+1}^n)'(a)}{(T_{k+1}^n)'(b)} \le \exp \left\{ c_1 | T_{k+1}^n(K) |\right\},$$
where $c_1 = \frac{D}{1-r}$ depends only on $\alpha$. After integration with respect to $b$, we find $|(T_{k+1}^n)'(a)| \le \frac{| T_{k+1}^n(K) |}{|K|}  \exp \left\{c_1 |T_{k+1}^n(K)|\right\}$, from which we deduce $$P_{k+1}^n {\bf 1}_K \ge {\bf 1}_{T_{k+1}^n(K)} \frac{|K|}{|T_{k+1}^n(K)|} \exp \left\{-c_1 |T_{k+1}^n(K)|\right\}.$$

If this new iterate of $K$ intersects the intermittent region, we consider the situation \ref{case:inter_1} or \ref{case:inter_2}, and continue the algorithm. If it is still in the hyperbolic region, but now contains the point $\frac{2}{3}$, we proceed in the following way. Let us call $L = T_{k+1}^n(K)$ the new iterate, and $L_l$ and $L_r$ the parts of the interval at the left and the right respectively of $\frac 2 3$. Either $| L_l | > \frac{1}{3} |L|$ or $| L_r | > \frac{1}{3} |L|$.

In the first case, after one iteration, the image of $L_l$ will be contained in $[\frac{2}{3}, 1]$, with the right extremity at 1. So after say $m$ steps, it will cover the whole unit interval, and the distortion is well controlled during this iteration. We then have $$\begin{aligned} P_{k+n+1}^{m+1} {\bf 1}_L & \ge P_{k+n+1}^{m+1} {\bf 1}_{L_l} \\ &\ge {\bf 1}_{T_{k+n+1}^{m+1}(L_l)} \frac{|L_l|}{|T_{k+n+1}^{m+1}(L_l)|} \exp \left\{- c_1 | T_{k+n+1}^{m+1}(L_l)| \right\} \\ & \ge \frac{1}{3}  |L| \exp\left\{ - c_1 \right\}.\end{aligned} $$

Setting $n_1 = n+m+1$, we thus have $$ \begin{aligned} P_{k+1}^{n_1} {\bf 1}_K = P_{k+n+1}^{m+1} P_{k+1}^{n} {\bf 1}_K & \ge P_{k+n+1}^{m+1} {\bf 1}_L \frac{|K|}{|L|} \exp \left\{ -c_1 |L| \right\}\\ & \ge \frac{1}{3} |K| \exp\left\{ - c_1 - c_1 |L|\right\}. \end{aligned}$$

In the second case, if the right part is longer than the left part, after one iteration, the iterate of the right part will be of the form $[0,x]$, and we fall in the case \ref{case:inter_2} of the algorithm. We can apply the control on the distortion given in the case \ref{case:inter_2} to $L_r$. Like in the previous case, doing this, we will get a factor $\frac 1 3$, but as we will see, the case \ref{case:inter_2} leads to the end of the algorithm, so we will meet the discontinuity point $\frac 2 3$ at most one time during the whole procedure. Hence the factor $\frac 1 3$ will appear only one time, and will not multiply itself several times, which could have spoiled the estimate.

~\\
{\bf Case \ref{case:inter_1}.} $K$ is included in an interval $I_{\ell,+}^k$. Since $T_{k+1}^{\ell}(I_{\ell,+}^k) = [a_2^{k+\ell}, 1 ] \subset [a_2^{0},1]$, after exactly $\ell$ iterations, the image of $K$ will be included in the hyperbolic region, and we continue with the case \ref{case:hyperb}. During this period of time, the distortion is controlled using the Koebe principle, that we recall below :

\begin{Lemma}[{Koebe Principle \cite[Theorem IV.1.2]{DeMelo}}] For all $\tau > 0$, there exists $C = C(\tau) > 0$ such that for all increasing diffeomorphism $g$ of class $C^3$ with a non-positive Schwarzian derivative \footnote{i.e.  $\frac{g'''(x)}{g'(x)} - \frac 3 2 \left( \frac{g''(x)}{g'(x)} \right)^2 \le 0$},
for all subintervals $J_1 \subset J_2$ such that $g(J_2)$ contains a $\tau$-scaled neighborhood of $g(J_1)$  \footnote{i.e. the intervals on the left and on the right of $g(J_1)$ in $g(J_2)$ have length at least $\tau |g(J_1)|$}, one has $$\frac{g'(x)}{g'(y)} \le \exp \left\{ C \frac{|g(x) - g(y)|}{|g(J_1)|}\right\} \text{ for all } x,y \in J_1.$$
\end{Lemma}

We apply it to $g$ defined as the composition of the analytic extensions to $(0, +\infty)$ of the left branches of $T_{k+\ell}, \ldots, T_{k+1}$ with $J_1 = I_{\ell,+}^k$ and $J_2 = [\delta, 2]$, where $\delta = \delta(k,\ell)$ is chosen small enough so that $\delta < a_{\ell+2}^k$ and $T_{0}^\ell \delta < \frac 1 2 a_2^{0}$. $g$ has non-positive Schwarzian derivative since it is a composition of maps that have non-positive Schwarzian derivatives.

We have $g(J_1) = [a_2^{k+\ell}, 1 ] \subset [a_2^{0},1]$ and $g(J_2) = [T_{k+1}^{\ell} \delta, g(2)] \supset [T_{0}^\ell \delta, 2] \supset [\frac{1}{2} a_2^{0}, 2]$.

Set $\tau = \min \left\{ \frac{a_2^{0}}{2 (1 - a_2^{0})}, \frac{1}{1 - a_2^{0}} \right\}$, which does not depends on the composition of maps, nor the number of steps $\ell$.

The interval at the left of $g(J_1)$ in $g(J_2)$ contains $[\frac{1}{2 } a_2^{0}, a_2^{0}]$, and thus has length longer than $\frac 1 2 a_2^{0} \ge \tau ( 1 - a_2^{0}) \ge \tau | g(J_1)|$. Similarly, the interval at the right of $g(J_1)$ in $g(J_2)$ contains $[1,2]$ and thus has length longer than $1 \ge \tau ( 1 - a_2^{0}) \ge \tau |g(J_1) |$.
We have proved that $g(J_2)$ contains a $\tau$-scaled neighborhood of $g(J_1)$, so the Koebe principle implies there exists $C = C(\tau)$ such that for all $a, b \in J_1$ one has  $$\frac{(T_{k+1}^{\ell})'(a)}{(T_{k+1}^{\ell})'(b)} \le \exp \left\{ C \frac{|T_{k+1}^{\ell}(a) - T_{k+1}^{\ell}(b)|}{|T_{k+1}^{\ell}(J_1)|} \right\} \le \exp \left\{ c_2 |T_{k+1}^{\ell}(a) - T_{k+1}^{\ell}(b)| \right\},$$ with $c_2 = 3 C(\tau)$ since $T_{k+1}^{k+\ell}(J_1) \supset [\frac 2 3, 1]$.

As $K \subset J_1$, one has for all $a,b  \in K$ $$\frac{(T_{k+1}^{\ell})'(a)}{(T_{k+1}^{\ell})'(b)} \le \exp \left\{ c_2 | T_{k+1}^{\ell}(K) |\right\},$$ which implies $$P_{k+1}^{\ell} {\bf 1}_K \ge {\bf 1}_{T_{k+1}^{\ell}(K)} \frac{|K|}{| T_{k+1}^{\ell} (K)|} \exp \left\{ -c_2 | T_{k+1}^{\ell} (K)| \right\}.$$
~ \\
{\bf Case \ref{case:inter_2}.} If more than one third of $K$ is in $[\frac 2 3, 1]$ and is of the form $[a,1]$, then we consider $K \cap [\frac 2 3, 1]$ and case \ref{case:hyperb} will hold until we cover the whole interval, and we lose a factor $\frac 1 3$ (only one time). Otherwise, we define $\ell_-$ as the least integer such that $I_{\ell_-}^k$ is included in $K$.
We consider two sub-cases according to whether the part of $K$ at the right of $a_{\ell_-}^k$ is of length at least $\frac{|K|}{3}$ or not.

In the first sub-case, we set $K' = K \cap [a_{\ell_- +1}^k, 1]$, which satisfies $|K'| \ge \frac{|K|}{3}$. Since $K' \subset I_{\ell_- -1, +}^k$ and $T_{k+1}^{\ell_- -1}(K') \supset I_1^{k+\ell_- -1}$, we have by the step \ref{case:inter_1}:  $$P_{k+1}^{\ell_- -1} {\bf 1}_K \ge \frac{1}{3} {\bf 1}_{I_1^{k+\ell_- -1}} |K| e^{-c_2}.$$

Since $T_{k+\ell_-}^2 I_1^{k+\ell_- -1} = [0,1]$, and $|T_\beta ' (x) |$ is bounded from above uniformly in $\beta$ and $x$ by some constant $M>0$, we find $$P_{k+1}^{\ell_- +1} {\bf 1}_K \ge \frac{1}{3M^2} |K| e^{-c_2}.$$

In the second sub-case, we choose $K'$ in such a way that $|K'| \ge \frac{|K|}{3}$, the right extremity of $K'$ is $a_{\ell_-}^k$ and the left extremity is to the right of $0$. We cut $K'$ into pieces $I_{\ell_-}^k, \ldots, I_{\ell_+}^k$ such that their union is of length longer than $\frac{|K'|}{3} \ge \frac{|K|}{9}$, with $\ell_+$ minimal. This choice to cut $K'$ rather than $K$ allows us to estimate $\ell_+$: indeed, if we set $K=[a, a_{\ell_-}^k]$, since $\ell_+$ is minimal, the length of $[a, a_{\ell_+ -1}^k]$ is at least $\frac{2 |K'|}{3} \ge \frac{2 |K|}{9}$. Hence, $\frac{C}{(\ell_+ -1)^{\frac 1 \alpha}} \ge a_{\ell_+ -1}^k \ge a_{\ell_+ -1}^k - a \ge \frac{2|K|}{9}$ and consequently $\ell_+ = \mathcal{O}(|K|^{- \alpha})$.

By the computation done for the case \ref{case:inter_1}, we have

$$\begin{aligned} P_{k+1}^{\ell_++1} {\bf 1}_K &\ge \sum_{\ell=\ell_-}^{\ell_+} P_{k+1}^{\ell_++1} {\bf 1}_{I_{\ell}^k}= \sum_{\ell=\ell_-}^{\ell_+} P_{k+\ell+1}^{\ell_+ - \ell +1} P_{k+1}^{\ell} {\bf 1}_{I_{\ell}^k} \\ & \ge \sum_{\ell=\ell_-}^{\ell_+} P_{k+\ell+1}^{\ell_+ - \ell +1} {\bf 1}_{T_{k+1}^{\ell}(I_{\ell}^k)} \frac{|I_{\ell}^k|}{|T_{k+1}^{\ell}(I_{\ell}^k)|} \exp \left\{ -c_2 |T_{k+1}^{\ell}(I_{\ell}^k)| \right\}. \end{aligned} $$

Since $T_{k+1}^{\ell}(I_{\ell}^k) =  [\frac 2 3, 1]$, which is sent after one iteration onto the whole interval, we have thanks to Remark \ref{R1} item \ref{R1_item2} $$ \begin{aligned} P_{k+1}^{\ell_++1} {\bf 1}_K \ge \sum_{\ell=\ell_-}^{\ell_+} P_{k+\ell+1}^{\ell_+ -\ell+1} {\bf 1}_{[\frac 2 3, 1]} \frac{|I_{\ell}^k|}{1/3} \exp \left\{ - \frac{c_2}{3} \right\}  & \ge \frac{c_3}{3} \exp\left\{- \frac{c_2}{3}\right\} \sum_{\ell=\ell_-}^{\ell_+} |I_{\ell}^k| \\ &\ge \frac{c_3}{27} \exp\left\{- \frac{c_2}{3}\right\} |K|, \end{aligned}$$ with $c_3$ the constant given in Remark \ref{R1}, since $P_{k+ \ell+1}^{\ell_+-\ell+1} {\bf 1}_{[\frac{2}{3}, 1]} \ge \frac{1}{3} P_{k+\ell+2}^{\ell_+-\ell} {\bf 1} \ge \frac{c_3}{3}$.

 ~\\
{\bf Conclusion.} Let $J$ be an interval of length at least $2 \epsilon$. We associate to $J$ a sequence of integers $n_1, m_1, n_2, m_2,  \ldots, n_p$ such that for $n_1$ steps the iterates of $J$ is in the hyperbolic region (with possibly $n_1 = 0$), then for $m_1$ steps, it is in situation of the case \ref{case:inter_1}, then it is again for $n_2$ steps in the hyperbolic region (recall that from the case \ref{case:inter_1}, we can only fall into the case \ref{case:hyperb}), and so on, until one iterate of $J$ crosses the singularity $\frac{2}{3}$, or case \ref{case:inter_2} happens. These two situations lead to the end of the algorithm. We will only consider the situation where case \ref{case:inter_2} happens, and when the part of $K$ to the right of $a_{l_-}^l$ has length not more than $\frac{|K|}{3}$, the others being similar.

For $n \ge n_1 + m_1 + \ldots + n_p  + \ell_+ + 1$, we have

$$ \begin{aligned}
&P_1^n {\bf 1}_J  \ge P_{n_1 + \ldots + \ell_+ + 2}^{n - n_1 - \ldots - \ell_+ -1} P_{n_1 + m_1 + \ldots + n_p +1}^ { \ell_+ + 1} \ldots P_{n_1+1}^{m_1} P_1^{n_1} {\bf 1}_J& \\
&\ge P_{n_1 + \ldots + \ell_+ + 2}^{n - n_1 - \ldots - \ell_+ -1} P_{n_1 + m_1 + \ldots + n_p +1}^ { \ell_+ + 1}  \ldots P_{n_1+1}^{ m_1} {\bf 1}_{T_1^{n_1}(J)} \frac{|J|}{|T_1^{n_1}(J)|}e^{ - c_1 |T_1^{n_1}(J)|}& \\
& \ge P_{n_1 + \ldots + \ell_+ + 2}^{n - n_1 - \ldots -  \ell_+ -1} P_{n_1 + m_1 + \ldots + n_p +1}^ {n_1+ \ldots + n_p + \ell_+ + 1}  \ldots  {\bf 1}_{T_1^{n_1+m_1}(J)}  \frac{|T_1^{n_1}(J)|}{|T_1^{n_1+m_1}(J)|} \frac{|J|}{|T_1^{n_1}(J)|} &
\\ & \times e^{ - c_1 |T_1^{n_1}(J)|   - c_2 |T_1^{n_1+m_1}(J)|}& \\
& \ge ... &\\
& \ge (P_{n_1+\ldots + \ell_+ +2}^{n - n_1 - \ldots - \ell_+ -1} {\bf 1}) \frac{c_3}{27} |T_1^{n_1+ \ldots +n_p} (J)| \frac{|T_1^{n_1+ \ldots + m_{p-1}} (J)|}{|T_1^{n_1+ \ldots +n_p} (J)|} \ldots   \frac{|T_1^{n_1}(J)|}{|T_1^{n_1+m_1}(J)|} \frac{|J|}{|T_1^{n_1}(J)|}&
\\& \times e^{ - c_1 |T_1^{n_1}(J)|  - c_2 |T_1^{n_1+m_1}(J)| - \ldots - c_1 |T_1^{n_1+ \ldots + n_p }(J)| - \frac{c_2}{3}  }&
\\ & \ge (P_{n_1+\ldots + \ell_+ +2 }^{n - n_1 - \ldots - \ell_+ -1} {\bf 1}) \frac{c_3}{27} |J| e^{- 2(\frac{c_2}{3} + c_1) (r^{n_2+ \ldots + n_p} + \ldots + r^{n_2})} \ge \frac{c_3^2}{27} e^{- \frac{2(\frac{c_3}{3} + c_1) r}{1 -r}} |J| =: \gamma |J|.
\end{aligned}
$$
~\\

One has to estimate the supremum over all possible values of $t = n_1+ m_1+ \ldots +n_p + l_+$ and shows it is of order $\epsilon^{- \alpha}$. Let $n_\epsilon'$ the minimal time needed for an interval of length at least $2 \epsilon$ to cover all the circle. We claim that $n_\epsilon ' = \mathcal{O}(\epsilon^{-\alpha})$, which concludes the proof since $n_1 + m_ + \ldots + n_p \le n_\epsilon '$, as after these iterations, $J$ has not covered the circle, and $l_+ = \mathcal{O}(\epsilon^{- \alpha})$, as we showed previously.

It remains to prove the claim. Since the first derivatives of all the $T_{\b}$ is strictly increasing on the circle, the minimal time associated to intervals $J$ of size $2\eps$, will be attained when an iterate of $J$ will be located around $0$, then moving  according to  case \ref{case:inter_2}. We first consider an iterate whose length is one third of that  of $J$ (see above),   located in  $(0, 2\eps/3)$: we call this {\em situation F}.  This implies $a_{d+t}^1 \le 2\eps/3$ which in turn shows the time needed to cover the circle is $n_\epsilon '' = [\frac{3c_{\a}}{2\eps}]^{\a}.$ Take now $J$ far from $0$; if in $n_{\eps}''$ steps it will not meet the point $2/3$, it will cover the circle, since the derivatives will be continuous along the path. Otherwise if it will meet $2/3$ in a number of steps $<n_{\eps} ''$,  the worst successive situation is to be  sent in $0$ in the situation F. In conclusion, the minimal time associated to intervals $J$ of size $2\eps$ will be bounded from above by $2n_{\eps}''.$

\section*{Acknowledgments}
RA and SV were supported by the ANR-
Project {\em Perturbations} and by
the PICS ( Projet International de Coop\'eration Scientifique), {\em Propri\'et\'es statistiques des syst\`emes dynamiques d\'eterministes et al\'eatoires}, with the University of Houston, n. PICS05968. RA was supported by Conseil G\'en\'eral Provence-Alpes-C\^ote d'Azur. SV thanks the University of Houston for supporting his visits during the preparation  of this work and the Newton Institute in Cambridge where this paper was completed during the Workshop {\em Mathematics for Planet Earth}. MN was supported by NSF grant DMS 1101315.  AT was partially supported by the Simons Foundation grant 239583. The authors thank the referee for the useful comments. We are extremely grateful to S\'ebastien Gou\"{e}zel who pointed out a gap in the distortion argument of a preliminary version and suggested how to correct it by using the Koebe principle.

\section*{Note added in proof}
A more careful analysis shows that in the proof of Property~({\bf P}) the monotonicity of the derivatives is not necessary to estimate $n_\epsilon$. Thus, Theorem 2.6 holds for more general maps than (\ref{PM}), e.g. those in \cite{LSV}; the details can be found in \cite{aimino_thesis}.


\begin{thebibliography}{99}

\bibitem{aimino_thesis} R. Aimino, Vitesse de m\'elange et th\'eor\`emes limites pour les syst\`emes dynamiques al\'eatoires et non-autonomes, Ph. D. Thesis, Universit\'e de Toulon, (2014)

\bibitem{AR} R. Aimino, J. Rousseau, Concentration inequalities for sequential dynamical systems of the unit interval, preprint

\bibitem{BH} W. Bahsoun, Ch. Bose, Y. Duan, Decay of correlation for random intermittent maps, \emph{Nonlinearity,} \textbf{27} (2014),  1543-1554

\bibitem{CR} J.-P. Conze, A. Raugi, Limit theorems for sequential expanding dynamical systems on $[0, 1]$, {\em Ergodic
theory and related fields}, {\bf 89121}, Contemp. Math., 430, Amer. Math. Soc., Providence, RI, 2007

\bibitem{DeMelo} W. de Melo, S. van Strien, {\em One-dimensional Dynamics}, Springer, Berlin, 1993

 \bibitem{GO} S. Gou\"{e}zel,  Central limit theorem and stable laws for intermittent maps {\em Probab. Theory Relat. Fields},
{\bf 128} 82–122, (2004)

\bibitem{GOT}C. Gupta, W. Ott, A. T\"or\"ok, Memory loss for time-dependent piecewise expanding systems in higher dimension, {\em Mathematical Research Letters}, {\bf 20},  (2013), 155-175

\bibitem{HH} H Hu, Decay of correlations for piecewise smooth maps with indifferent fixed points, {\em Ergodic Theory  and Dynamical Systems}, {\bf 24}, (2004), 495-524

\bibitem{LSV} C. Liverani, B. Saussol, S. Vaienti, A probabilistic approach to intermittency,   \textit{Ergodic theory and dynamical systems},  {\bf 19}, (1999), 671-685

\bibitem{OSY} W. Ott, M. Stenlund, L.-S. Young, Memory loss for time-dependent dynamical systems, {\em  Math. Res.
Lett.}, {\bf 16} (2009), pp. 463-475.

\bibitem{Sarig} O. Sarig,  Subexponential decay of correlations, {\em Invent. Math.},  {\bf 150} 629-653, (2002)

\bibitem{SS}W. Shen, S. Van Strien,  On stochastic stability of expanding circle maps with neutral fixed points, {\em Dynamical Systems, An International Journal},  {\bf 28}, (2013)

\bibitem{S} M. Stenlund, Non-stationary compositions of Anosov diffeomorphisms, {\em Nonlinearity} {\bf 24} (2011) 2991-3018

\bibitem{SYZ}M. Stenlund, L-S. Young, H. Zhang, Dispersing billiards with moving scatterers, {\em Comm. Math. Phys.}, to appear (2013)













		



    \end{thebibliography}
\end{document}